\documentclass[12pt]{amsart}

\usepackage{amssymb}

\usepackage{enumerate}

\usepackage{graphicx}

\usepackage[all,cmtip]{xy}

\usepackage{xcolor}

\makeatletter
\@namedef{subjclassname@2010}{%
  \textup{2010} Mathematics Subject Classification}
\makeatother

\usepackage[T1]{fontenc}

\usepackage[normalem]{ulem}

\newtheorem{thm}{Theorem}[section]
\newtheorem{cor}[thm]{Corollary}
\newtheorem{lem}[thm]{Lemma}
\newtheorem{prop}[thm]{Proposition}

\theoremstyle{definition}

\newtheorem{rem}[thm]{Remark}

\newtheorem{question}[thm]{Question}

\numberwithin{equation}{section}

\newcommand{\cl}{\operatorname{cl}}

\newcommand{\cov}{\operatorname{cov}}
\newcommand{\vint}{\operatorname{int}}
\newcommand{\id}{\operatorname{id}}

\begin{document}


\baselineskip=17pt



\title{Inverse systems with simplicial bonding maps and cell structures}
\author{Wojciech D\k{e}bski}
\author{Kazuhiro Kawamura}
\address{University of Tsukuba,Tsukuba, Japan}
\email{kawamura@math.tsukuba.ac.jp}

\author{Murat Tuncali}
\address{ Mathematics and Computer Science, Nipissing University, North Bay, Canada}
\email{muratt@nipissingu.ca}
\author{E.D. Tymchatyn}
\address{E. D. Tymchatyn, Dept. of Maths. and  Stats., University of Saskatchewan, 
106 Wiggins Rd. Saskatoon Sk. S7N5E6 Canada}
\email{tymchat@math.usask.ca}
\date{July 21, 2019}
\begin{abstract}
For a topologically complete spaceneed $X$ and a family of closed covers $\mathcal A$ of $X$ satisfying a `` local refinement condition'' and a ``completeness condition,'' we give a construction of an inverse system $\boldmath{N}_{\mathcal A}$ of simplicial complexes and simplicial bonding maps such that the limit space $N_{\infty} = \varprojlim \boldmath{N}_{\mathcal A}$ is homotopy equivalent to $X$.
A connection with cell structures \cite{DebskiTymchatynMetr}, \cite{DebskiTymchatyn} is discussed.

\end{abstract}
\thanks{The second author is supported by JSPS KAKENHI Grant Number 17K05241. 
The third named author is partially supported by National Science and Engineering Research Council Discovery Grant.}

\setcounter{footnote}{0}

\maketitle

\section{Introduction}

The notion of inverse limits and its variants provides a useful scheme for studying topological spaces by polyhedra.  For every topological space $X$, there associates an inverse 
system  ${\bf X}= ( X_{\lambda},\pi^{\mu}_{\lambda}:X_{\mu} \to X_{\lambda}; \Lambda ) $ of polyhedra $X_\lambda$ and continuous maps $\pi^{\mu}_{\lambda}$ that reflects the topology of $X$.  Such a system is typically given by a polyhedral resolution in the sense of \cite{MardesicSegal} which yields an HPol-expansion defining the shape type of $X$. In most cases, the bonding maps $\pi_{\lambda}^{\mu}$ are continuous or piecewise linear: if the polyhedra $X_\lambda$'s are equipped with specific triangulations and the bonding maps are required to be simplicial with respect to the triangulations, then the inverse limit 
$\varprojlim \mathbf{X}$ may not be  homeomorphic, but closely related, to $X$.
This has been made explicit by 
Dydak in \cite[Theorem 2.8]{Dydak}: for each compact metrizable space $X$, there exist an inverse sequence $\bold X$ of finite simplicial complexes and simplicial bonding maps, and 
a hereditary shape equivalence $f:\varprojlim \bold X \to X$.  
See also \cite{RubinTonic}.

This note studies a similar construction for non-compact spaces.  For a topologically complete space $X$ and a family $\mathcal A$ of {\it closed} locally finite normal coverings of $X$ satisfying a `` local refinement condition'' and a `` completeness condition,'' we construct an inverse system ${\bf N}_{\mathcal A} = ( N_{\lambda}, \pi^{\mu}_{\lambda}:\Lambda) $ of simplicial complexes and simplicial bonding maps, a proper map 
$\pi:N_{\infty}:= \varprojlim{\bf N}_{\mathcal A} \to X$, and a continuous map $p:X \to N_{\infty}$ such that $\pi \circ p = \mathrm{id}_{X}$ and $p \circ \pi$ is $\pi$-fiberwise homotopic to $\mathrm{id}_{N_\infty}$. In particular $N_{\infty}$ and $X$ have the same homotopy type.  Also the family $\mathcal A$ naturally defines an inverse system ${\bf F}_{\mathcal A} = (F_{\lambda}, \pi^{\mu}_{\lambda}:\Lambda)$, where each 
$F_{\lambda}$ is a flag complex which contains $N_{\lambda}$ as its subcomplex
and each $\pi_{\lambda}^{\mu}$ is a simplicial map, so that
$N_{\infty} = \varprojlim {\bf F}_{\mathcal A}$.  Since flag complexes are uniquely determined by their 1-skeletons, we could say that the space $N_\infty$ is induced by an inverse sequence of graphs and graph homomorphisms.

The above result gives a connection of inverse systems of simplicial bonding maps with cell-structures of topological spaces.
Cell structures first appeared in unpublished notes of D\k{e}bski in 1990. They are  inverse systems of vertex sets of graphs and graph homomorphisms (see \cite{DebskiTymchatynMetr}  and \cite {DebskiTymchatyn}).  Every topological space admits a cell structure that can  be thought of as a system of discrete approximations of the space.
It was proved in \cite[Theorem 2]{DebskiTymchatyn} that cell structures determine topologically complete spaces and continuous maps.
The above inverse system  ${\bf F}_{\mathcal A} = ( F_{\lambda}, \pi^{\mu}_{\lambda}:\Lambda) $ naturally defines a cell-structure. When the space $X$ is compact Hausdorff, then the system of projections ${\bf p} = (p_{\lambda}:N_{\infty} \to F_{\lambda})$ forms an HPol-expansion of the space $N_{\infty}$ and thus ${\bf p}\circ p = (p_{\lambda}\circ p)$ is an HPol-expansion of $X$.  Thus we could say that an HPol-expansion of a compact Hausdorff space is obtained from a cell-structure.

Our construction is motivated by a theorem of Marde{\v s}i{\' c} \cite{MardesicAPR} (see also \cite[Chap.I, Section 6.4, Theorem 7]{MardesicSegal}) stating that every topological space admits a polyhedral resolution.  
For each 
topologically complete space $X$ with a family $\mathcal A$ of closed locally finite normal covers of $X$ satisfying the above-mentioned conditions,  we modify the construction to obtain inverse systems ${\bf N}_{\mathcal A} = (N_{\lambda}, \pi^{\mu}_{\lambda}:N_{\mu} \to N_{\lambda})$ and 
$\bf{F}_{\mathcal A}$ with the desired properties.

Throughout, for a subset $S$ of a topological space $Y$, $\vint(S)$ and $\cl(S)$ denote the interior and the closure of $S$ in $Y$. 
For a cover \(\alpha \) of  \(Y\),  let $\text{star}_{\alpha}(S) =  
\{V \in \alpha~|~S \cap V \neq \emptyset \}$ and define inductively
\[
\mathrm{star}^{n+1}_{\alpha}(S) = 
\{ V \in \alpha~|~\mbox{there exists}~ W \in \text{star}_{\alpha}^{n}(S)~\mbox{such that}~V \cap W \neq \emptyset \}.
\]

Simplicial complexes are assumed to be endowed with specific triangulations and the weak topology with respect to the triangulations.  
For a simplicial complex $N$, $N^{(i)}$ denotes the $i$-skeleton of $N$. 
In particular $N^{(0)}$ is the set of the vertices of $N$. 
For a simplex $\sigma$ of $N$, 
$\text{Int}(\sigma)$ denotes the interior of $\sigma$, the points of $\sigma$ whose barycentric coordinates are all positive.
Following \cite{DavisMoussong} we say that a simplicial complex $F$ is a {\it flag complex} if every subset $\{v_{1},\ldots,v_{n}\}$ of $F^{(0)}$ with the property: $v_{i}, v_{j}$ span a 1-simplex for each $i \neq j$,  spans a simplex of $F$.  
The first barycentric subdivision of each simplicial complex is a flag complex.
The reader may consult \cite{EngelkingGT} for standard results in general topology 
and \cite{MardesicSegal} for results in shape theory.

\section{Construction of polyhedral inverse systems with simplicial bonding maps}
\label{sec-constr}

Recall that a {\it topologically complete space}  is a Tychonoff space which is complete in its finest uniformity. A closed locally finite cover $\alpha$ is called {\it normal}  if there exists a partition of unity $\Phi_{\alpha}=\{\phi_{ \alpha,V}\mid V\in\alpha \}$  subordinated to $\alpha$: 
each $\phi_{\alpha,V}:X \to [0,1]$ is continuous, satisfies
$\mathrm{supp}(\phi_{\alpha,V})= \mathrm{cl}(\phi_{\alpha,V}^{-1}((0,1])) \subset \mathrm{int}(V)$, and $\sum_{V\in \alpha}\phi_{\alpha,V} = 1$.
 In the remainder of this section we fix an arbitrary topologically complete space \(X\)
and fix a family  \(\mathcal A \) of closed, locally finite,  normal  covers of \(X\)  satisfying the following two conditions:

\begin{enumerate}

\item[I)] For each open subset \(U\) of  \(X\) and for each \(x\in U\), there exists \(\alpha\in A\) such that \(\bigcup\mathrm{star}_{\alpha}(x)\subset U\).
 \item[II)] For each $\alpha \in \mathcal{A}$ select a member $f(\alpha) \in \alpha$.  
If the collection \(\{f(\alpha)\}\) has the finite intersection property, then \(\bigcap _{\alpha\in A} f(\alpha)\ne \emptyset.\)
\end{enumerate}
It should be mentioned here that the family $\mathcal A$ need not be cofinal with respect to the refinement-order: a closed locally finite normal covering $\gamma$ of $X$ may not admit $\alpha \in {\mathcal A}$ that refines $\gamma$.  the condition I),  a ``local refinement'' requirement, is a substitute for the cofinality.

\begin{prop}
\label{prop-closed-her}
If \(Y\) is a closed subset of \(X\) then the conditions I) and II) are satisfied for the family 
\( \mathcal{A}|Y = \{ \alpha |Y ~|~\alpha \in \mathcal{A}\} \), where
$\alpha|Y = \{F\cap Y~|~F \in \alpha\}$.
\end{prop}
\begin{proof}
The condition I)  is  clearly satisfied. Let \(f:A\rightarrow \bigcup A\) be a selection such that the sets \(f(\alpha)\cap Y\) have the finite intersection property. Then by II) \(\bigcap_{\alpha\in A} f(\alpha)\ne \emptyset\). Let \(x\in \bigcap_{\alpha\in A} f(\alpha)\). Assume \(x\notin Y\). Then by I) there exists \(\alpha\in A\) such that \(f(\alpha)\subset X\setminus Y\). Thus, \(f(\alpha)\cap Y=\emptyset\). This is a contradiction.
\end{proof}

\subsection{Construction}
Here we construct an inverse system $(F_{\lambda},\pi_{\lambda}^{\mu};\Lambda)$ with the limit space $F_\infty$ and a proper map $\pi:F_{\infty} \to X$ that is a homotopy equivalence.

Let \(\Lambda\) be the directed set of all finite subsets 
of \(\mathcal{A}\) ordered by inclusion.  It is a cofinite set in that each element of \( \Lambda \) has only finitely many predecessors.
For an element \(\lambda=\{\alpha_1,\ldots,\alpha_n\}\) of $\Lambda$, let 
\[
N_{\lambda}^{(0)} = \{v=(V_1,...,V_n)\in \alpha_1\times\cdots\times \alpha_n\mid V_1 \cap\ldots\cap V_n \neq \emptyset \},
\]
and for  \(v=(V_1,\cdots,V_n)\in N^{(0)}_{\lambda}\), let \(\wedge v=V_1\cap\ldots\cap V_n\).
Define the collection \(\cov_{\lambda} \) by
\[
\cov_{\lambda} = \{ \wedge v\mid v\in N^{(0)}_{\lambda} \}
\] 
which is a closed, locally finite, normal cover of \(X\) indexed by 
the set \(N_{\lambda}^{(0)}\); the local finiteness follows from that of $\alpha_{1},\ldots,\alpha_{n}$.  
For the existence of a partition of unity subordinated to \(\cov_{\lambda}\), let \(\{\phi_{\alpha_{i},V}~|~V\in \alpha_{i}\}\) be a partition of unity subordinated to \(\alpha_{i}\).  
For each \(v \in N_{\lambda}^{(0)}\)  with 
$v = (V_{1},\ldots,V_{n})$, let
\begin{equation}\label{eq:partition}
\phi_{v} = \phi_{\alpha_{1},V_{1}}\cdots \phi_{\alpha_{n},V_{n}}.
\end{equation}
It is readily verified that \(\{\phi_{ v}~|~ v \in N^{(0)}_{\lambda}\)\} is a partition of unity subordinated to the cover \(\cov_{\lambda}\) indexed by $N^{(0)}_{\lambda}$.
Note that it may be the case $\wedge v = \wedge w$ for distinct elements  $v$ and $w$ of $N_{\lambda}^{(0)}$.  

Let \(F_{\lambda}\)  be the flag complex defined as follows: the vertex set of \(F_{\lambda}\) is the set \(N_{\lambda}^{(0)}\); and a finite set \(\{v_1,v_2,...,v_n\}\), \(v_{i}\in 
N^{(0)}_\lambda \), spans a simplex of \(F_\lambda\) if and only if \(\wedge v_i \cap \wedge v_j\ne\emptyset\) for each pair \(i,j\in\{1,...,n\}\).
For each point $a$ of $F_\lambda$, there exists a unique simplex 
\(\sigma_{\lambda}(a)\) of $F_{\lambda}$ such that  
$a \in \text{Int} \sigma_{\lambda}(a)$. 
If \(\sigma_{\lambda}(a)\) has the vertex set 
\( \{v_{1},\ldots,v_{n}\}\) of $F_{\lambda}^{(0)}$,  then let \(\wedge a\) be the subset of $X$ (maybe empty)  defined by 
\[
\wedge a = \wedge v_1 \cap \cdots \cap\wedge v_k .
\] 
Let \(N_\lambda\)  be the subcomplex of  \(F_\lambda\) defined as follows:
the set of vertices is equal to $N_{\lambda}^{(0)}$ and a finite set \(\{v_1,v_2,...,v_n\}\) of 
vertices spans a simplex in \(N_\lambda\) if and only if \(\bigcap_{i=1}^n\wedge v_i \ne\emptyset\). 
By definition,  we have
$N_{\lambda}^{(1)} = F_{\lambda}^{(1)}$. Also for a point $a \in F_{\lambda}$, 
 \(\wedge a\not=\emptyset\) if and only if \(a\in N_{\lambda}\).
The complex $N_{\lambda}$ is almost the same as the nerve complex $\mathcal{N}(\cov_{\lambda})$ of \(\cov_\lambda\), except that distinct vertices of $N_{\lambda}$ may define the same vertex of $\mathcal{N}(\cov_{\lambda})$.  There exists a natural simplicial 
homotopy equivalence $N_{\lambda} \to \mathcal{N}(\cov_{\lambda})$ such that the inverse image of each simplex of $\mathcal{N}(\cov_{\lambda})$ is a simplex. 

For two elements $\lambda \leq \mu$ of $\Lambda$, written as 
\(\lambda = \{\alpha_{1},\ldots, \alpha_{n}\}$ and $\mu=\{\alpha_1,\ldots,\alpha_n,\ldots,\alpha_m\}$,  let 
\(\pi^{\mu}_{\lambda}:F_{\mu}\rightarrow F_{\lambda}\) be the simplicial map which sends each vertex \((V_1,\ldots,V_n,\ldots,V_m) \in \alpha_{1}\times \cdots \times \alpha_{m}\) of \(F^{(0)}_{\mu}\) to the vertex 
\((V_1,\ldots,V_n)\) of \(F^{(0)}_{\lambda}\).  We see
$
\pi^{\nu}_{\lambda}=\pi^{\mu}_{\lambda}\circ\pi^{\nu}_{\mu},~~
\lambda\leq\mu \leq \nu,
$
and obtain an inverse system \({\bf F}_{\mathcal A}= (F_{\lambda},\pi_{\lambda}^{\mu};\Lambda)\).  Due to the inclusion \(\pi^{\mu}_{\lambda}(N_{\mu}) \subset N_\lambda\), we have a subsystem
\({\bf N}_{\mathcal A} = (N_{\lambda},\pi_{\lambda}^{\mu};\Lambda)\).
Let    \(F_{\infty}=\varprojlim \bf{F}_{\mathcal A}\), 
\(N_{\infty}=\varprojlim \bf{N}_{\mathcal A}\).  For \(\lambda \in \Lambda\),
let \(\pi_{\lambda}:F_{\infty}\rightarrow F_{\lambda}, \pi_{\lambda}:N_{\infty}\rightarrow N_{\lambda},\) be the \(\lambda^{\text{th}}\) coordinate projections. 
For a point 
\(z\in F_\infty\),  the point \(\pi_{\lambda}(z)\) is also denoted by \(z(\lambda)\).

\begin{lem} \label{lem-n-cov-int-onepointN}
Under the above notation, we have the following.
\begin{itemize}
\item[(1)]  For each open subset $U$ of \(X\), for each \(x\in U\) and for each positive integer \(n\), there exists \(\lambda \in \Lambda\) such that \(
\bigcup\emph{star}^n_{\cov_\lambda}(x)\subset U\).
\item[(2)]  For each \(z\in F_{\infty}\), the set 
\(\bigcap_{\lambda\in\Lambda}(\wedge z(\lambda)) \) contains exactly one point.
\end{itemize}
\end{lem}
\begin{proof}
(1) We use induction on $n$. The case \(n=1\) follows from the condition I).
Assume that (1) holds for $(n-1)$ and apply the condition I) to take $\alpha\in {\mathcal A}$ such that 
$\bigcup\text{star}_{\cov_{\alpha}}(x)\subset U$.  
By the local finiteness of the covering $\cov_{\{\alpha\}}$ there exists an open neighborhood $V$ of $x$ such that
\[
w \in \cov_{\{\alpha\}},~w \cap V \neq \emptyset \Rightarrow x \in w.
\]
Use the induction hypothesis to find a  \(\lambda\in\Lambda\) such that
 \(\bigcup\text{star}^{n-1}_{\lambda}(x)\subset V\).
We show
\[
\bigcup\text{star}^{n}_{\cov_{\lambda \cup \{\alpha\}}}(x)\subset U.
\]
For each $c \in \text{star}^{n}_{\cov_{\lambda \cup \{\alpha\}}}(x)$, there exists
$d \in \text{star}^{n-1}_{\cov_{\lambda \cup \{\alpha\}}}(x)$ such that 
$c \cap d \neq \emptyset$.  The elements $c$ and $d$ are written as
\[
c = u \cap a,~d = v \cap b, ~~u,v \in \cov_{\lambda}, a,b \in \alpha.
\]
Then $v$ is in $\text{star}^{n-1}_{\lambda}(x)$ and hence 
$v \subset V$.  Thus the choice of $V$ and the inclusion
\[
a\cap V \supset u\cap v \supset c \cap d \neq \emptyset
\]
imply that $x \in a$.  Hence $a \subset \bigcup \text{star}_{\cov_{\lambda}}(x) \subset U$.
Thus $c = u \cap a \subset U$.

(2) Let \(z\) be a point of $F_{\infty}$. For $\lambda \in \Lambda$, the point 
$z(\lambda)$ is an interior point of the unique simplex
$\sigma_{\lambda}(z(\lambda))$ of $F_\lambda$. Let
$\{ v_{j}~|~j=1,\ldots,k \}$ be  the set of vertices of  
\(\sigma_\lambda(z(\lambda))\) and recall $\wedge z(\lambda) = \cap_{j=1}^{k}\wedge v_{j}$.
For each $\lambda, \mu \in \Lambda$ with $\lambda \leq \mu$ we have
$\pi^{\mu}_\lambda(\sigma_{\mu}(z(\mu))^{(0)})= \sigma_\lambda(z(\lambda))^{(0)}$, and 
each $\sigma_{\lambda}(z({\lambda}))^{(0)}$ is a finite set.
We thus have an inverse system $(\sigma_{\lambda}(z(\lambda))^{(0)}, \pi_{\lambda}^{\mu}; \Lambda)$ of finite sets and surjective bonding maps so that the limit space 
$\varprojlim \sigma_{\lambda}(z(\lambda))^{(0)}$ is nonempty and projects onto each $\sigma_{\lambda}(z(\lambda))^{(0)}$.  Hence
there exists \(z_{0}\in F_{\infty}\) such that \(z_{0}(\lambda)\ \in \sigma_{\lambda}^{(0)} \subset F_{\lambda}^{(0)} \) for each  \(\lambda\in\Lambda\).
Note that for \(\lambda<\mu\), \(\wedge z_{0}(\mu)\subset \wedge z_{0}(\lambda)\).
Let $\mathcal{F}_{0} = \{ \wedge z_{0}(\lambda)~|~\lambda \in \Lambda\}$.  By the above remark, it has the finite intersection property and thus by the condition II),  the intersection \(\bigcap_{\lambda\in \Lambda}\wedge z_{0}(\lambda)\)\ is non-empty.  
If two distinct points $p,q$ are in the set, we take an open neighborhood $U$ of $p$ which does not contain $q$ and apply the condition I) to find $\lambda \in \Lambda$ such that $\bigcup\text{star}_{\cov_\lambda}(p) \subset U$.  Since $p\in z_{0}(\lambda)$, we see $z_{0}(\lambda) \in \text{star}_{\cov_{\lambda}}(p)$ and thus $\wedge z_{0}(\lambda) \subset U$.
But then $q \in \wedge z_{0}(\lambda) \subset U$, a contradiction.  This implies that
\(\bigcap_{\lambda\in \Lambda}\wedge z_{0}(\lambda)\)\ is a singleton.
Let \(\{x\}=\bigcap_{\lambda\in\Lambda}\wedge z_{0}(\lambda)\).

Next we show that the above $x$ does not depend on the choice of $z_0$.
Let \(z_{1}\) be another point of $F_{\infty}$ such that \(z_{1}(\lambda)  \in \sigma_\lambda(z)^{(0)}\)  for each \(\lambda\).
Since \(z_{0}(\lambda)\) and \(z_{1}(\lambda)\) are vertices of 
$\sigma_{\lambda}(z(\lambda)) \in F_{\lambda}$,  we have \(\wedge z_{0}(\lambda)\cap \wedge z_{1}(\lambda) \neq \emptyset\).
Let $U$ be an arbitrary open neighborhood of $x$ and use the condition I) to find $\alpha \in \mathcal A$ such that $\bigcup \mathrm{star}_{\alpha}(x) \subset U$.  Let $\mu = \lambda \cup \{\alpha\}$ for which we have $\bigcup \mathrm{star}_{\cov_\mu}(x) \subset U$.
Since $x \in \wedge z_{0}(\mu)$ we have
$\wedge z_{0}(\lambda) \in \text{star}_{\cov_\mu}(x)$ and hence $\wedge z_{0}(\lambda) \subset U$.  We see
\[
U\cap \wedge z_{1}(\lambda) \supset U \cap \wedge z_{1}(\mu) \supset
\wedge z_{0}(\mu)\cap \wedge z_{1}(\mu) \neq \emptyset.
\]
Thus $x \in \cl(\wedge z_{1}(\mu)) = \wedge z_{1}(\mu)$.  Hence $\{x\} = \cap_{\lambda}(\wedge z_{1}(\lambda))$.

Finally we show $\{x\} = \cap_{\lambda}(\wedge z(\lambda))$.  For each vertex $\wedge v$ of 
$\sigma_{\lambda}(z(\lambda))$, we choose $z_{0} \in F_{\infty}^{(0)}$ such that 
$z_{0}(\lambda) = \wedge v$ and $z_{0}(\mu) \in \sigma_{\mu}(z_{0}(\mu))^{(0)}$ for each 
$\mu$. We then have $x \in \wedge v$ by the above.  Thus $x \in \wedge z(\lambda)$ for each $\lambda \in \Lambda$, which completes the proof of (2).
\end{proof}

The space \(N_{\infty}=\varprojlim \bf{N}_{\mathcal A}\) is a closed subspace of $F_{\infty}$, 
but more holds.

\begin{prop}
\label{thm-equal-limits-2} 
We have the equality \( N_\infty = F_\infty\).
\end{prop}
\begin{proof}
Let \(z = (z(\lambda))_{\lambda} \in F_\infty\).
By Lemma 2.2 (2) we have $\bigcap_{\lambda} (\wedge z(\lambda))$ is a singleton $\{x\}$.  Then for $\lambda \in \Lambda$, each vertex of $\sigma_{\lambda}(z(\lambda))$ contains the point $x$ and thus $\sigma_{\lambda}(z(\lambda))$ is a simplex of the complex $N_\lambda$.  Hence $z(\lambda) \in N_{\lambda}$ and $z \in N_{\infty}$.
\end{proof} 
Since $N_{\infty} = F_{\infty}$ is the limit of flag complexes $F_\lambda$ which are uniquely determined by their 1-skeletons, we may say that the limit space $N_\infty$ is determined by an inverse system of graphs.  
Notice that the barycentric subdivision $\mathrm{sd} N_{\lambda}$ is always a flag complex 
and we thus obtain an inverse system
$(\mathrm{sd} N_{\lambda}, \mathrm{sd} \pi_{\lambda}^{\mu}; \Lambda)$ whose limit is homeomorphic to 
$N_{\infty}$, yet the graph $(\mathrm{sd} N_{\lambda})^{(1)}$ is not isomorphic to $N_{\lambda}^{(1)}$.  Here we obtain the flag complex $F_{\lambda}$ with $F_{\lambda}^{(1)}= N_{\lambda}^{(1)}$.

\bigskip
\noindent

In view of Lemma 2.2 and Proposition 2.3, we define a map $\pi:N_{\infty} \to X$ by
\[
\{ \pi(z) \} = \cap_{\lambda}(\wedge z(\lambda))
\]
for $z = (z(\lambda))_{\lambda} \in N_{\infty} = F_{\infty}$.  It is convenient to introduce the space $N_{\infty}^{(0)} = \varprojlim(N_{\lambda}^{(0)},\pi_{\lambda}^{\mu}|N_{\mu}^{(0)};\Lambda)$.

\begin{prop}
\begin{enumerate}
\item[(1)] The map $\pi$ is continuous.
\item[(2)] For points $z,z' \in N_{\infty}^{(0)}$, $\pi(z') = \pi(z)$ if and only if $\wedge z'(\lambda) \in \emph{star}_{\cov_\lambda}(\wedge z(\lambda))$ for each $\lambda \in \Lambda$.
\item[(3)] For points $z \in N_{\infty}$ and $z_{0} \in N_{\infty}^{(0)}$ such that
$z_{0}(\lambda) \in \sigma_{\lambda}(z(\lambda))^{(0)}$ for each $\lambda$, we have 
$\pi(z) = \pi(z_{0})$.
\end{enumerate}
\end{prop}

\begin{proof}
(1)  Take a point $z \in N_{\infty}$ and let $U$ be an open neighborhood of $\pi(z)$.  
Take 
$\lambda $ such that $\bigcup\text{star}_{\cov_\lambda}(\pi(z)) \subset U$ by Lemma 2.2 (1).
Let 
$V_{\lambda} = \cup\{ \text{Int} \sigma~|~z(\lambda) \in \sigma \in N_{\lambda}\}$
be the open star of $z(\lambda)$ in $N_{\lambda}$.  It is an open neighborhood of $z(\lambda)$ and we show that
\[
\pi(\pi_{\lambda}^{-1}(V_{\lambda})) \subset U.
\]
Let $w \in \pi_{\lambda}^{-1}(V_{\lambda})$.  For each $\lambda$, we see  
$\sigma_{\lambda}(w(\lambda))$ contains $z(\lambda)$.  Thus $\sigma_{\lambda}(z(\lambda))$ is a face of $\sigma_{\lambda}(w(\lambda))$.  This implies
$\wedge w(\lambda) \subset \wedge z(\lambda)$ and thus
$\pi(w) \in \wedge w(\lambda) \subset \wedge z(\lambda)$.
Hence
\[
\pi(w) \in \wedge z(\lambda) \subset \bigcup \text{star}_{\cov_\lambda}(\pi(z)) \subset U
\]
and the conclusion follows.

(2)  Suppose  \(\pi(z)=\pi(z')\) for \(z,z'\in N^{(0)}_{\infty}\) and let \(x=\pi(z)\).
For each $\lambda$ we see \(x\in \wedge z(\lambda)\cap \wedge z'(\lambda)\). Thus \(\wedge z(\lambda)\cap \wedge z'(\lambda)\not=\emptyset\) and $\wedge z'(\lambda) \in \text{star}_{\cov_\lambda}(\wedge z(\lambda))$ for each \(\lambda\in\Lambda\).

Conversely let \(\wedge z'(\lambda)\in\text{star}_{\cov_\lambda}(\wedge z(\lambda))\) for each \(\lambda\in\Lambda\).
Let \(x=\pi(z)\), \(x'=\pi(z')\) and assume \(x'\not=x\).
Take disjoint open sets \(U\) and \(U'\) containing \(x\) and \(x'\) respectively.
By I) there are \(\alpha, \alpha' \in\mathcal{A}\) such that \(\bigcup\mathrm{star}_{\alpha}(x)\subset U\) and \(\bigcup\mathrm{star}_{\alpha}(x')\subset U'\).
Since \(x\in\wedge z(\{\alpha \}) \subset U \) and \(x'\in\wedge z(\{\alpha'\})\subset U'\), we have \(\wedge z(\{\alpha\})\cap\wedge z'(\{\alpha'\})= \emptyset\).
Thus \(\wedge z(\lambda)\cap\wedge z'(\lambda)=\emptyset\) for \(\lambda=\{\alpha,\alpha'\}\), a contradiction.

(3) Let \(z\in N_{\infty}\) and \(z_{0} \in N^{(0)}_\infty\) such that \(z_{0}(\lambda)\) is a vertex of the simplex \(\sigma_{\lambda}(z(\lambda))\) for each \(\lambda\in\Lambda\).
Since \(\wedge z({\lambda})\subset\wedge z_{0}({\lambda})\) we have
 \(\pi(z)\in\bigcap_{\lambda\in\Lambda} (\wedge z_{0}(\lambda))\).
Thus, \(\pi(z)=\pi(z_{0})\).
\end{proof}

\smallskip

To study the map $\pi$ further, let us introduce a map $p:X \to F_{\infty} = N_{\infty}$ as follows: for each $\lambda = \{\alpha_1,\ldots,\alpha_n\}\in\Lambda$, we define
\(p_{\lambda}:X \to N_{\lambda}\) as the canonical map associated with the 
partition of unity (\ref{eq:partition}).  It has the property
\begin{equation}\label{eq:canonical}
p_{\lambda}^{-1}(\cup \text{star}_{\cov_\lambda}(\wedge v)) \subset \wedge v
\end{equation}
for each $\wedge v \in \cov_\lambda$.  As is proved in \cite[p. 85]{MardesicSegal}, 
we have \(p_{\lambda}=\pi^{\mu}_{\lambda}\circ p_{\mu}\) for \(\lambda\leq\mu \) and hence obtain a system of maps ${\bf p} = (p_{\lambda}) :X\to {\bf N}_{\mathcal A} = (N_{\lambda},\pi_{\lambda}^{\mu};\Lambda )$
which induces the limit map $p = \varprojlim \mathbf{p}:X \to N_{\infty}$.
The map $p$ is continuous and satisfies 
\(p_{\lambda}=\pi_{\lambda}\circ p\)  for each $\lambda$.

\begin{thm}\label{thm:pi}
Under the above notation, we have the following.
\begin{itemize}
\item[(1)] The map $\pi:N_{\infty} = F_{\infty} \to X$ is a perfect map. \item[(2)] We have
\begin{enumerate}
\item[(2a)] $\pi \circ p = \id_{X}$, and
\item[(2b)] there exists a homotopy $H:N_{\infty}\times [0,1] \to N_{\infty}$ such that $H_{0} = \id_{N_\infty}, H_{1}= p\circ \pi$ and $\pi \circ H = \pi \circ \mathrm{proj}$, where 
$\mathrm{proj}:N_{\infty}\times [0,1] \to N_{\infty}$ denotes the projection onto $N_{\infty}$.
In particular each fiber of $\pi$ is contractible.
\end{enumerate}
\end{itemize}
\end{thm}
\begin{proof}
(1) Fix an arbitrary point $x \in X$.  We first prove that 
$\pi^{-1}(x)$ is non-empty and compact. 
For each $\lambda=\{\alpha_{1},\ldots,\alpha_{n}\}$, let
\begin{equation}\label{eq:Clambda}
C_{\lambda} = \{ w \in N_{\lambda}^{(0)}~|~x \in \wedge w\}
\end{equation}
which is a finite subset of $N_{\lambda}^{(0)}$ due to the local finiteness of the covers \(\alpha_1,\ldots\alpha_n\).
If $\mu \geq \lambda$ then \(x\in\wedge\pi_{\lambda}^{\mu}(u)\) and \(\pi_{\lambda}^{\mu}(u)\in C_{\lambda}\) for each \(u\in C_{\mu}\).
Thus
$(C_{\lambda},\pi_{\lambda}^{\mu}|C_{\mu};\Lambda)$ forms an inverse system of finite sets and hence has the nonempty inverse limit.  If $z = (z(\lambda)) \in \varprojlim C_{\lambda}$, then $x \in \wedge z(\lambda)$ and thus $\sigma_{\lambda}(z(\lambda))^{(0)} \subset C_{\lambda}$ for each $\lambda$.  By the definition of the map $\pi$, we have $x = \pi(z)$. Next let $K_{\lambda}$ be the simplex of $N_{\lambda}$ spanned by $C_{\lambda}$.
We prove
\begin{equation}\label{eq:fiber}
\pi^{-1}(x) = \varprojlim K_{\lambda}.
\end{equation}
If $\pi(z) = x$, then $x \in z(\lambda)$ for each $\lambda$, which implies that $\sigma_{\lambda}(z(\lambda))^{(0)} \subset C_{\lambda}$ and hence $z(\lambda) \in K_{\lambda}$. Thus we have $z \in \varprojlim K_{\lambda}$ and $\pi^{-1}(x) \subset \varprojlim K_{\lambda}$.  The reverse inclusion is straightforward. 
This proves the above and hence $\pi^{-1}(x)$ is compact. 

Next we show that $\pi$ is a closed map.  Let $G$ be a closed subset of $N_{\infty}$ and take an arbitrary point $x \in \cl(\pi(G))$.  Under the above notation, we consider the simplex $K_{\lambda}$, $\lambda \in \Lambda$, for the point $x$.
First observe that $G_{\lambda} := K_{\lambda}\cap \cl(\pi_{\lambda}(G))$ is compact as a closed subset of the simplex \(K_{\lambda}\).
Let  \(U=X\setminus\bigcup\{\wedge v\mid x\not\in \wedge v, v \in N^{(0)}_{\lambda}\}\).
It is an open set containing $x$.  
It follows that there exists $z \in G$ such that
$\pi(z) \in U$. This implies $\wedge z(\lambda) \cap U \neq \emptyset$.
For each vertex $v \in \sigma_{\lambda}(z(\lambda))$, we have 
$x \in \wedge v$ by the definition of $U$ and thus $z(\lambda) \in K_{\lambda}\cap\pi_{\lambda}(G) \subset G_{\lambda}$ and \(G_{\lambda}\not=\emptyset\).  
It follows that $\pi_{\lambda}^{\mu}(G_{\mu}) \subset G_{\lambda}$ and hence $\varprojlim (G_{\lambda},\pi_{\lambda}^{\mu})$ is a nonempty compact subset of $N_{\infty}$.  Let \(z'\in\varprojlim (G_{\lambda},\pi^{\mu}_\lambda)\).
Then we have \(z'\in G\):  if not, there exists an open neighborhood $O$ of $z'$ such that $G \cap O = \emptyset$ because $G$ is closed.  Take an index $\lambda$ and an open neighborhood $V$ of $z(\lambda)$ such that
$\pi_{\lambda}^{-1}(V) \subset O$.  Then we see that $z'(\lambda) \notin \cl\pi_{\lambda}(G)$, a contradiction.
Also \(x\in \wedge z'(\lambda)\) for each \(\lambda\in\Lambda\).
Thus we see \(x=\pi(z')\) and hence $\pi(G)$ is closed.
This proves (1).

(2) For each $x \in X$ and for each $\lambda$, take any vertex $v \in \sigma_{\lambda}(p_{\lambda}(x))$.  We see $p_{\lambda}(x)  \in \text{star}_{\cov_\lambda}(\wedge v)$ and by (\ref{eq:canonical}) $x \in p_{\lambda}^{-1}(\text{star}_{\cov_\lambda}(\wedge v)) \subset \wedge v$.  Hence we have $x \in \wedge p_{\lambda}(x)$ for each
$\lambda$.  It follows from the definition of $\pi$ that $\pi (p(x)) = x$. This proves the statement (2.a).
To prove the statement (2.b), 
we define a homotopy \(H:N_{\infty}\times I\rightarrow N_{\infty}\) as follows: 
for a point $z \in N_{\infty}$ and $\lambda$, both $z(\lambda) = \pi_{\lambda}(z)$ and 
$p_{\lambda}(\pi(z))$ are points of the simplex $\sigma_{\lambda}(z(\lambda))$.  Thus
\[
H_{\lambda}(z,t) =t p_{\lambda}(\pi(z))+(1-t)\pi_{\lambda}(z)
\]
is a well-defined point of $\sigma_{\lambda}(z(\lambda))$.
Also we have $H_{\lambda}(z,t) = \pi_{\lambda}^{\mu}(H_{\mu}(z,t))$ for 
$\lambda \leq \mu$ because $\pi_{\lambda}^{\mu}$ is a simplicial map. Hence $H(z,t) = (H_{\lambda}(z,t))$ is a well-defined point of 
$N_\infty$. This defines a continuous homotopy between $\id_{N_\infty}$ and $p\circ \pi$.  In order to verify $\pi \circ H = \pi$, we take $z \in N_{\infty}$ and set $x = \pi(z)$.  By the definition of $\pi$, we see $x \in \wedge z(\lambda)$. Let $C_{\lambda}$ be the finite subset defined in (\ref{eq:Clambda}) for $x$ and let $K_\lambda$ be the simplex spanned by $C_\lambda$.  Each vertex of $\sigma_{\lambda}(z(\lambda))$ belongs to $C_{\lambda}$, from which it follows that $\sigma_{\lambda}(z(\lambda)) \subset K_{\lambda}$.  The equality (\ref{eq:fiber}) shows that $z \in \varprojlim \sigma_{\lambda}(z(\lambda)) \subset \pi^{-1}(x)$.  Hence we have $\pi \circ H(z,t) = \pi(z)$. This proves (2.b).
\end{proof}

We have a system of inclusions $\mathbf{i}= (i_{\lambda}:N_{\lambda}\hookrightarrow F_{\lambda})$.  Proposition \ref{thm-equal-limits-2} does not guarantee that $\mathbf i$ is an isomorphism in pro-Top (see \cite{MardesicSegal}).  When the space $X$ is paracompact and $\mathcal A$ is a cofinal family, it is indeed the case.

\begin{thm}
\label{prop-inclusion-F-N-nv}
Assume that \(X\) is a paracompact space  and \(\mathcal{A}\) be a cofinal subfamily of all closed, locally finite, normal covers of \(X\). Then for each $\lambda \in \Lambda$ there exists $\mu \in \Lambda$ with $\mu \geq \lambda$ such that
$\pi_{\lambda}^{\mu}(F_{\mu}) \subset N_{\lambda}$. 
In particular 
the system of the inclusions $(N_{\lambda} \hookrightarrow F_{\lambda}; \lambda \in \Lambda)$ is an isomorphism in the category pro-Top.
\end{thm}
\begin{proof}
Let $\lambda = \{ \alpha_{1},\ldots, \alpha_{n}\}$.  
For each point $x$ of $X$, 
let \(U_x=X\setminus\bigcup\{F\in\cov_{\lambda}\mid x\not\in F\}\).
Take a locally finite open star refinement \(\mathcal{U}\) of \(X\) of \(\{U_x\mid x\in X\}\) and its closed normal locally finite refinement $\mathcal V$.
There is a cover $\alpha \in \mathcal A$ which refines $\mathcal V$.
Let $\mu = \{\alpha_{1},\ldots, \alpha_{n},\alpha\} \geq\lambda$.
Then $\cov_{\mu}$ refines the cover \(\alpha \) and hence \(\mathcal U\).
Let \(v_1,\ldots,v_m\) be the vertices of an simplex of \(F_{\mu}\).
Then \(\wedge v_i\cap\wedge v_j\not=\emptyset\), \(i,j=1,\ldots,m\).
Since each \(\wedge v_i\) is contained in some \(U\in\mathcal{U}\) and \(\mathcal U \) is a  star refinement of the cover \(\{U_x\mid x\in X\}\), we find a point $x\in X$ such that 
$U_x$ contains all \(\wedge v_{i},~i=1,\ldots m\).
Thus \(x\in\wedge\pi^{\mu}_{\lambda}(v_i)\) for \(i=1,\ldots,m\) 
and \(\{\pi_{\lambda}^{\mu}(v_{1}), \ldots, \pi_{\lambda}^{\mu}(v_{m})\}\) spans a simplex of 
$N_\lambda$.  Thus $\pi_{\lambda}^{\mu}(F_{\mu}) \subset N_{\lambda}$.
The last statement follows from the first and Morita's Lemma \cite[Chap. II, Section 2,2, Theorem 5]{MardesicSegal}.
\end{proof}

\subsection{Polyhedral expansions}

\bigskip
Every polyhedral resolution (\cite{MardesicSegal}) of a space $X$ gives an HPol-expansion (\cite{MardesicSegal}) defining the shape type  of $X$.  Likewise the above inverse systems $(N_{\lambda},\pi_{\lambda}^{\mu};\Lambda)$ and $(F_{\lambda},\pi_{\lambda}^{\mu};\Lambda)$ yield polyhedral expansions of spaces 
under some mild assumption on the collection $\mathcal A$.  The expansions so obtained have two features: (i) the polyhedra $N_\lambda$ or $F_\lambda$ have specific triangulations and the bonding maps are simplicial with respect to these triangulations, while the resolution constructed in \cite[Chap.1, Section 6.4, Theorem 7]{MardesicSegal} has continuous, but not necessarily simplicial, bonding maps.  (ii) for the systems $(N_{\lambda},\pi_{\lambda}^{\mu};\Lambda)$ and
$(F_{\lambda},\pi_{\lambda}^{\mu};\Lambda)$, the equality $\pi_{\lambda}^{\nu} = \pi_{\lambda}^{\mu}\circ \pi_{\mu}^{\nu}$ holds for $\lambda \leq \mu \leq \nu$, while in the classical  {\v C}ech system on the basis of refinement-relation, 
 $\pi_{\lambda}^{\nu}$ is only homotopic to $\pi_{\lambda}^{\mu}\circ \pi_{\mu}^{\nu}$.

Here, rather than recalling the notion of HPol-expansions, we follow the formulation of \cite{MoritaCCoh} (see also \cite[Chap.I, section 2.4]{MardesicSegal}).  For spaces $Y$ and $K$, 
$[Y,K]$ denotes the set of homotopy classes of maps $Y \to K$.  
Every continuous map $g:X \to Y$ induces a function $g^{\#}:[Y,K] \to [X,K]$. 
The inverse system 
$(N_{\lambda},\pi_{\lambda}^{\mu};\Lambda)$ yields a direct system
$([N_{\lambda},K],(\pi_{\lambda}^{\mu})^{\#};\Lambda)$ of sets which defines the direct limit(or colimit) $\varinjlim ([N_{\lambda},K],(\pi_{\lambda}^{\mu})^{\#};\Lambda)$ in the category {\bf Sets} of sets and functions.
The above map $p_{\lambda}:X \to N_{\lambda}$ induces a function  
$p^{\#}_{\lambda}:[N_{\lambda},K] \to [X,K]$ for each $\lambda$ which yields the limit map
$p^{\sharp}:\varinjlim ([N_{\lambda},K],(\pi_{\lambda}^{\mu})^{\#};\Lambda) \rightarrow [X,K]$.
In the sequel, $K$ is assumed to be a simplicial complex with the weak topology. Since the classes of spaces that are homotopy equivalent to
\begin{itemize}
\item[(i)] simplicial complexes with the weak topology,
\item[(ii)] simplicial complexes with the metric topology,
\item[(iii)] CW complexes,
\item[(iv)] metric ANR's,
\end{itemize}
all coincide (\cite{Milnor}, \cite{MardesicSegal} etc),  we may assume that $K$ belong to any of the above classes. The system 
$(p_{\lambda}:X \to N_{\lambda})$ is an HPol-expansion of $X$ 
(see \cite{MardesicSegal}), or a system satisfying (M1) and (M2) of
\cite[p.256]{MardesicANR}, precisely when the above $p^{\#}$ is a bijection for each simplicial complex $K$.
Shape theory for general topological spaces was studied in \cite{MoritaSh} where it was shown that $\text{Sh}(X) = \text{Sh}(\tau X) = \text{Sh}(\mu X)$ where 
$\tau$ denotes the Tychonoff functor \cite{MoritaCCoh} and $\mu X$ denotes the completion with respect to the finest uniformity of $\tau X$.  Thus there is no loss of generality in assuming that the spaces are topologically complete in the next results.

\begin{prop}
\label{prop-bijection-top}
Let \(X\) be a topologically complete space and \(\mathcal{A}\) be a cofinal subfamily of all closed, locally finite, normal covers of \(X\). Then $\mathcal A$ satisfies the conditions I) and II).  Further we have a bijection
\begin{equation}\label{eq:directlim1}
p^{\sharp}:\varinjlim ([N_{\lambda},K],(\pi_{\lambda}^{\mu})^{\#};\Lambda) \rightarrow [X,K]
\end{equation}
for each simplicial complex \(K\).  
If \(X\) is a paracompact space, then
\begin{equation}\label{eq:directlim2}
p^{\sharp}:\varinjlim ([F_{\lambda},K],(\pi_{\lambda}^{\mu})^{\#};\Lambda) \rightarrow [X,K]
\end{equation}
is also a bijection for each simplicial complex \(K\).
\end{prop}

\begin{proof}
First notice that the family of locally finite open normal covers is cofinal in the family of normal covers.
Indeed as was shown in \S2 of \cite{MoritaCCoh},
for an open normal cover \(\mathcal{U}\) in \(X\) there exist a continuous map \(f:X\rightarrow Y\) to a metric space \(Y\) and an 
open cover \(\mathcal{V}\) of \(Y\) such that the cover 
\( f^{-1}(\mathcal{V})\) refines \(\mathcal{U}\) (see also \cite[Chap.I, section 6.2, Lemma 1]{MardesicSegal}).  By the paracompactness of $Y$, there exists a locally finite open cover \(\mathcal{W}\) of \(Y\) that refines \(\mathcal{V}\). 
Then \(f^{-1}(\mathcal{W})\) is a locally finite, open, normal cover of \(X\) refining cover \(\mathcal{U}\).

We show that $\mathcal A$ satisfies the conditions I) and II). 
Let \(U\subset X\) be an open set and \(x\in U\).
There exists a continuous function  \(\phi:X\rightarrow [0,1]\) such that \(\phi(x)=1\), 
$\mathrm{supp}(\phi) := \mathrm{cl}\phi^{-1}((0,1]) \subset U$, and 
\(\phi \equiv 1\) on a neighborhood of \(x\).
Therefore \(\{U,X\setminus\{x\}\}\) is a normal cover with partition of unity \(\{\phi,1-\phi\}\).
Take a normal open cover \(\mathcal{U}\) which is a star refinement of \(\{U,X\setminus\{x\}\}\).
By the cofinality of \(\mathcal{A}\) and the remark at the beginning of the proof, there exists an \(\alpha\in\mathcal{A}\) which refines \(\mathcal{U}\), and thus \(\alpha\) is a star refinement of \(\{U,X\setminus\{x\}\}\).
This implies  the condition I).  To verify the condition II), let 
$\{f(\alpha)~|~ \alpha\in\mathcal{A}\}$ be a family of subsets with the finite intersection property such that $f(\alpha)\in\alpha$ for each $\alpha \in \mathcal A$.  By the cofinality of $\mathcal A$ and the cofinality of the  family of the locally finite normal open coverings in that of the open normal coverings, the collection $\{f(\alpha)~|~\alpha\in {\mathcal A}\}$ contains arbitrarily small sets with respect to normal open covers.  By recalling that the finest uniformity is generated by the family of normal open covers of $X$ (\cite[Chap.8, 8.1.C]{EngelkingGT}), we conclude that 
$\bigcap f(\alpha) \neq \emptyset$ by the topological completeness of $X$.

The bijection (\ref{eq:directlim1}) is a consequence of \cite[Theorem 4.3]{MoritaCCoh}. 
Notice, that the use of open/closed covers causes no essential difference since 
each of the family of closed/open covers involved is cofinal in the other.
Notice also that the nerve complex $\mathcal{N}(\cov_{\lambda})$ may be replaced with the complex $N_{\lambda}$ because they are homotopically equivalent by a natural homotopy equivalence $N_{\lambda}\to \mathcal{N}(\cov_{\lambda})$.

If \(X\) is a paracompact then the second conclusion (\ref{eq:directlim2})
follows from Theorem
\ref{prop-inclusion-F-N-nv}.   
\end{proof}
Note that a closed cover $\mathcal G$ of a paracompact space $X$ with $\bigcup \{ \mathrm{int}G~|~G \in {\mathcal G}\} = X$ is a normal cover. 

A inverse system $\mathbf X$ of polyhedra with continuous bonding maps may fail to be an HPol-expansion of the limit space $\varprojlim {\mathbf X}$ when the limit space is not compact (\cite[Chap. I, Section 6, Example 1]{MardesicSegal}).  It follows from Proposition \ref{prop-bijection-top} that the inverse systems ${\mathbf N}_{\mathcal A}$ and 
${\mathbf F}_{\mathcal A}$ constructed above are HPol-expansions of the limit space  $N_{\infty}=F_{\infty}$ under a suitable assumption on the space $X$ and the family $\mathcal A$.  

\begin{cor}
\label{cor-bijection-top}
Let $X$ be a topologically complete space and let $\mathcal A$ be a family of closed covers as in Proposition \ref{prop-bijection-top}.  Then
we have a bijection
\[
\pi^{\sharp}:\varinjlim ([N_{\lambda},K],(\pi_{\lambda}^{\mu})^{\#},\Lambda) \rightarrow [N_\infty,K],
\]
for each simiplicial complex $K$.  If moreover \(X\) is paracompact, then the same holds for $(F_{\lambda},\pi_{\lambda}^{\mu};\Lambda)$.

\end{cor}

For compact spaces we obtain the same conclusion under a weaker hypothesis.
\begin{prop}
\label{prop-bijection-topcomp}
If \(X\) is a compact Hausdorff space and \(\mathcal{A}\) is a family satisfying the condition I), then $\mathcal A$ is cofinal in the family of  
all closed, locally finite, normal covers of \(X\). Thus we obtain the bijections 
of Proposition \ref{prop-bijection-top}.  In particular  \(\bf{N}_{\mathcal A}\) and \(\bf{F}_{\mathcal A}\) are HPol-expansions of the limit space 
$N_{\infty} = F_{\infty}$.
\end{prop}
\begin{proof}
We show the first statement. The rest follows from Proposition \ref{prop-bijection-top}.
For a finite closed normal cover \(\alpha\) of \(X\), we show that 
there is $\lambda \in \Lambda$ such that, for each $x\in X$, $\mathrm{star}_{\lambda}(x) \subset \mathrm{int}F$ for some $F\in \alpha$.
This implies that $\lambda$ refines $\alpha$.  Suppose that 
\(F_{\lambda}=\{x\in X\mid \bigcup\mathrm{star}_{\lambda}(x)\not\subset \mathrm{int}(F)~\forall F\in\alpha \} \ne \emptyset\) for each $\lambda$. Each \(F_{\lambda}\) is a closed set and \(F_{\mu}\subset F_{\lambda}\) for \(\mu\geq\lambda\). Then \(\bigcap F_{\lambda}\not=\emptyset\) and take a point \(x\in\bigcap F_{\lambda}\).
Thus we have  \(x\in\vint(F)\) for some \(F\in\alpha\).
Then, by the condition I), we obtain \(\bigcup\mathrm{star}_{\alpha}(x)\subset \vint(F)\) for some \(\lambda\in\Lambda\), a contradiction. 
\end{proof}

\section{Cell structures and polyhedral expansions}

A \emph{graph} is an ordered pair  \((G,r)\) of a (possibly infinite) discrete set \(G\) and a symmetric and reflexive binary relation \(r\) on \(G\).  Two elements $a,b \in G$ are said to be {\it adjacent} if $(a,b) \in r$.  For \(a\in G\) let \(\text{star}_r(a) =\{b \in G|(a,b)\in r\}\). A graph homomorphism \(f:(G,r)\rightarrow (G',r')\) between graphs, that is a function $f:G \to G'$ satisfying \((a,b)\in r\) implies  \((f(a),f(b))\in r'\), is termed as a \emph{continuous map} here. For an inverse system of graphs $\mathbf{G} = (G_{\lambda},\pi_{\lambda}^{\mu};\Lambda)$ we continue to use the notation of Section 2: let $G_{\infty} = \varprojlim(G_{\lambda},\pi_{\lambda}^{\mu};\Lambda)$ and for \(\lambda\),  
let \(\pi_{\lambda}:G_{\infty}\rightarrow G_{\lambda}\) be the \(\lambda^{th}\) coordinate projection.  For a point $z \in G_{\infty}$, the element $\pi_{\lambda}(z) \in G_{\lambda}$ is also denoted by $z(\lambda)$.
A \emph{cell-structure} \(\mathbf{G} = (G_{\lambda},\pi^{\mu}_{\lambda}:G_{\mu} \to G_{\lambda}; \Lambda)\) is an inverse system of  graphs and continuous bonding maps satisfying the conditions a) and b) below. 
\begin{enumerate}
\item[a)]
for each \(z\in G_{\infty}\) and \(\lambda\in\Lambda\), there exists \(\mu\geq \lambda\) such that
\[
\pi^{\mu}_{\lambda}(\text{star}_{r_{\mu}}^2(z(\mu))) \subset \text{star}_{r_{\lambda}}(z(\lambda)),
\]
\item[b)] for each \(z\in G_{\infty}\) and
  \(\lambda\in\Lambda\), there exists \(\mu\geq\lambda\) such that
 \(\pi^{\mu}_\lambda(\text{star}_{r_{\mu}}(z(\mu)))\) is finite.
\end{enumerate}
A point of $G_{\infty}$ is called a {\it thread}. 
Two threads \(z,z'\in G_{\infty}\) are \emph{equivalent}, written as \(z\sim z'\), if $z(\lambda)$ and $z'(\lambda)$ are adjacent in $G_{\lambda}$ for each \(\lambda\). 
By \cite[Theorem 3.4]{DebskiTymchatyn},
the relation \(\sim\) is a closed equivalence relation with compact equivalence classes and the quotient space $G^{\ast} = G_{\infty}/\sim$ with the quotient topology induced by $G_\infty$, is called the \emph{topological space represented by the cell structure} 
\({\mathbf G}= (G_{\lambda},\pi^{\mu}_{\lambda};\Lambda)\).  The quotient projection is denoted by $\pi_{\mathbf G}:G_{\infty} \to  
G^{\ast}$.  It is a perfect map because it is a closed map with fibers being compact equivalence classes.
A net \(y:\Lambda\rightarrow\bigcup_{\lambda}G_{\lambda}\) indexed by $\Lambda$ 
with \(y(\lambda)\in G_{\lambda}\) for each \(\lambda\) is called a \emph{Cauchy net} 
if there exists  \(\lambda\in\Lambda\) such that,  for each
  \(\lambda^0\), we have \(\pi^{\lambda^1}_{\lambda^0}(y(\lambda^{1}))\) and
   \(\pi^{\lambda^2}_{\lambda^0}(y(\lambda^{2}))\) are adjacent for all 
 \(\lambda^{1},\lambda^{2} \geq\lambda,\lambda^0\).
A Cauchy net \(y\) is said to \emph{converge} to a thread \(z\in G_{\infty}\) if there exists \(\lambda\in\Lambda\) such that \(y(\mu)\) and \(z(\mu)\) are adjacent for each \(\mu\geq\lambda\).
The cell structure is \emph{complete} if every Cauchy net converges to a thread.

The construction given in Section 2 yields cell structures associated with topologically complete spaces.
Let \(X\) be a topologically complete space and let \(\mathcal{A}\) be a set of  closed, locally finite, normal covers of  \(X\) which\(\) satisfies the conditions I)-II) of Section 2. 
Under the notation of Section 2, the set of vertices  \(N^{(0)}_\lambda= F^{(0)}_\lambda\), 
a discrete set, admits the symmetric and reflexive relation \(r_\lambda\)  defined by \((u,v)\in r_\lambda\), $u,v \in F_{\lambda}^{(0)}$, if and only if \(\wedge u \cap \wedge v \ne \emptyset\), or equivalently $u$ and $v$ span an 1-simplex of $F_{\lambda}$.
Thus the graph \((F^{(0)}_\lambda,r_\lambda) \) is isomorphic to the 1-skeleton of 
\( F_\lambda\).
The next theorem shows that the inverse system \(\mathbf{F}^{(0)}= (F^{(0)}_\lambda,\pi^{\mu}_\lambda,\Lambda)\) defined in Section 2, 
with the above relation $r_\lambda$ for each $\lambda \in \Lambda$, is a cell structure.  The 
space $F_{\infty}^{(0)} = \varprojlim \mathbf{F}^{(0)}$ with the quotient projection 
$\pi_{{\mathbf F}^{(0)}}:F_{\infty}^{(0)} \to F_{\infty}^{(0)}/\sim$ represent the space $X$ in a canonical way.
Part of this theorem was first proved in
 \cite[Theorem 8.1]{ DebskiTymchatyn}.

\begin{thm}
\label{thm-Complete-cell-structure}
Let $X$ be a topologically complete space and let \(\mathcal{A}\) be a cofinal subfamily of the family of all closed, locally finite, normal covers of \(X\).
Then the inverse system \(\mathbf{F}^{(0)}= (F^{(0)}_\lambda,\pi^{\lambda'}_\lambda,\Lambda)\) is a complete cell structure 
representing a space canonically homeomorphic to $X$: there exists a unique homeomorphism $h:X \to {(F^{(0)})}^{\ast}:= F_{\infty}^{(0)}/\sim$ from $X$ to the quotient space ${(F^{(0)})}^{\ast}$ defined by ${\mathbf F}^{(0)}$ such that 
$h\circ \pi|F_{\infty}^{(0)} = \pi_{{\mathbf F}^{(0)}}$.
\end{thm}

\begin{proof}
Recall that $\mathcal A$ satisfies the conditions I) and II) by Proposition 2.7. We show that $\mathbf{F}^{(0)}$ is a cell structure.  For this, let
\(z\in F^{(0)}_\infty\). By the definition of $\pi$ we have
\(\bigcap_{\lambda\in \Lambda}z(\lambda) =\{\pi(z)\}\). For \(\lambda\in \Lambda\),
let \( V = X\setminus \bigcup \{U \in \cov_\lambda|\pi(z)\notin  U\}\). 
Then \(V\) is an open neighborhood of \(\pi(z)\) and meets only finitely many elements of \(\cov_{\lambda}\). By Lemma \ref{lem-n-cov-int-onepointN} (1), there exists \(\lambda'\ge \lambda\) so
 \(\text{star}^3_{\cov_{\lambda'}}(\pi(z))\subset V\). 
 Then  \(\pi^{\lambda'}_\lambda(\text{star}^2_{\cov_{\lambda'}}(z(\lambda'))) \subset \pi^{\lambda'}_{\lambda}(\text{star}^3_{\cov_\lambda}(\pi(z)))\subset \text{star}_{\cov_\lambda}(z(\lambda))\). 
Thus \(\mathbf{F}^{(0)} = (F^{(0)}_\lambda,\pi^{\mu}_\lambda,\Lambda)\) satisfies the conditions a) and b) and is a cell structure.

In order to prove the completeness, let $(y(\lambda))_{\lambda \in \Lambda}$ be a Cauchy net with $y(\lambda) \in F_{\lambda}^{(0)} (\lambda \in \Lambda)$.  
We show 
\begin{equation*}
(\ast)~~\mbox{for each}~ \lambda, ~\mbox{there exists}~\mu \geq \lambda~\mbox{such that}~ A_{\lambda,\mu} := \{ \pi_{\lambda}^{\nu}(y(\nu))~|~\nu \geq \mu\}~\mbox{is a finite set,}
\end{equation*}
which implies that $(y(\lambda))_{\lambda}$ converges to a thread by 
 \cite[Lemma 3.7]{DebskiTymchatyn}.  
Take $\lambda^{-1} \in \Lambda$ such that for each $\lambda^{0},
\lambda^{1}, \lambda^{2}$ with $\lambda^{1}, \lambda^{2} \geq \lambda^{0}, \lambda^{-1}$ we have
\begin{equation}\label{eq:cauchy}
\pi_{\lambda^0}^{\lambda^1}(y(\lambda^{1}))~\mbox{and}~
\pi_{\lambda^0}^{\lambda^2}(y(\lambda^{2}))~\mbox{are adjacent, that is,}~
\wedge \pi_{\lambda^0}^{\lambda^1}(y(\lambda^{1})) \cap \wedge \pi_{\lambda^0}^{\lambda^2}(y(\lambda^{2})) \neq \emptyset.
\end{equation}
For $\lambda \in \Lambda$, let \(F_{\lambda} = \cl(\bigcup_{\mu\geq\lambda}\wedge y(\mu))\). The family $\{F_{\lambda}\}$ has the finite intersection property.
We first verify
\begin{equation*}
\begin{array}{ll}
(\ast\ast)~~\mbox{for each locally finite, open, normal cover}~\mathcal U ~\mbox{of}~X, ~\mbox{there exist}~ \lambda \in \Lambda ~\mbox{and}~  U \in \mathcal U \\
\quad \quad\quad~\mbox{such that}~ F_{\lambda} \subset U.
\end{array}
\end{equation*}
Indeed, for a locally finite open normal cover \(\mathcal{U}\) take an open normal cover \(\mathcal{V}\) which $\mathrm{star}^{2}$-refines \(\mathcal{U}\).  By taking a closed shrinking and by using the cofinality of $\mathcal A$, we have an $\alpha \in {\mathcal A}$ that refines $\mathcal V$.  Take $\lambda \geq \{\alpha\}, \lambda^{-1}$.
Since $\cov_{\lambda}$ refines $\mathcal V$, there exists $V \in \mathcal V$ such that 
$\wedge y(\lambda) \subset V$.  For each $\mu \geq \lambda$, we see, by (\ref{eq:cauchy}), 
$\wedge \pi_{\lambda}^{\mu}(y(\mu)) \cap \wedge y(\lambda) \neq \emptyset$.  Hence we have $\wedge y(\mu) \subset \wedge \pi_{\lambda}^{\mu}(y(\mu)) \subset \bigcup \mathrm{star}_{\cov_\lambda}(\wedge y(\lambda)) \subset \mathrm{star}_{\mathcal V}(V)$
and thus $F_{\lambda} \subset \mathrm{star}_{\mathcal V}^{2}(V) \subset U$
for some $U \in \mathcal U$.  This proves $(\ast\ast)$.

Recall that the collection of the normal open covers forms a base of the finest uniformity of the Tychnoff space $X$.
The topological completeness of $X$ together with $(\ast\ast)$ then imply that the family $\{F_{\lambda}\}$ contains arbitrarily small sets with respect to the finest uniformity. This together with the finite intersection property imply that the intersection $\bigcap_{\lambda \in \Lambda} F_{\lambda}$ contains a point $x$.
For the proof of $(\ast)$, take an arbitrary $\lambda \in \Lambda$ and choose an open neighborhood $U$ of $x$ such that $\mathrm{star}_{\cov_\lambda}(U)$ is a finite set 
by the local finiteness of $\cov_{\lambda}$.  Applying $(\ast\ast)$ to the cover $\{U, X\setminus\{x\}\}$ and noticing $x \in \bigcap F_{\nu}$, we find $\mu \geq \lambda$ such that $F_{\mu} \subset U$.  For each $\nu \geq \mu$ we have $\wedge y(\nu) \subset F_{\mu} \subset U$ which implies $\wedge \pi_{\lambda}^{\nu}(y(\nu)) \cap U \neq \emptyset.$  Thus we have $A_{\lambda,\mu} \subset \mathrm{star}_{\cov_{\lambda}}(U)$ and $A_{\lambda,\mu}$ is a finite set.  This proves $(\ast)$ and proves the completeness of ${\mathbf F}^{(0)}$.

The last statement follows from the definitions of the relation $\sim$, the maps $\pi$(see also Proposition 2.4 (2))  and 
$\pi_{{\mathbf F}^{(0)}}$. Recall that $\pi_{{\mathbf F}^{(0)}}$ is a perfect map and note that \(\pi|_{N^0_{\infty}}:N^{(0)}_{\infty}\rightarrow X\) is also a perfect map since \(N^{(0)}_{\infty}\) is a closed subset of \(N_{\infty}\).
\end{proof}

For a graph $(G,r)$, we define a flag complex ${\mathcal F}(G)$ as follows: the vertex set of ${\mathcal F}(G)$ is the set of vertices of $G$;  a set of vertices $\{v_{1},\ldots, v_{n}\}$ spans a simplex of ${\mathcal F}(G)$ if $v_{i},v_{j}$ are adjacent in $G$ for each $i,j=1,\ldots, n$.  Every graph homomorphism 
$f:G \to H$ induces a simplicial map ${\mathcal F}(f):{\mathcal F}(G) \to {\mathcal F}(H)$.

\begin{rem}
A cell structure \({\mathbf G} = (G_{\lambda},\pi_{\lambda}^{\mu}, \Lambda) \) that represents a space $Y$ induces
an inverse system ${\mathcal F}({\mathbf G}) = ({\mathcal F}(G_{\lambda}), {\mathcal F}(\pi_{\lambda}^{\mu}), \Lambda)$ of flag complexes. 
Also the equivalence relation $\sim$ on $G_{\infty}$ naturally extends to the one on $\varprojlim {\mathcal F}({\mathbf G})$ 
in such a way that the quotient space $\varprojlim {\mathcal F}({\mathbf G})/\sim$ is homeomorphic to $Y$.  Applying this to the cell structure \({\mathbf F}^{(0)}\) of Theorem 3.1 
with a suitable family $\mathcal A$,
we see that the system \({\mathcal F}({\mathbf F}^{(0)})\) is precisely the inverse system
\({\mathbf F}_{\mathcal A} = (F_{\lambda},\pi_{\lambda}^{\mu};\Lambda )\) constructed in Section \ref{sec-constr}, and the quotient map $\varprojlim {\mathcal F}({\mathbf F}^{(0)}) \to \varprojlim {\mathcal F}({\mathbf F}^{(0)}) /\sim$ is identified with the homotopy equivalence $\pi:F_{\infty} \to X$.
In particular, Proposition 2.7, Corollary 2.8  and Proposition 2.9 hold for the system \({\mathcal F}({\mathbf F}^0)\).  
Thus a cell structure $\mathbf G$ of a paracompact space $X$ with a suitable family of closed normal covers $\mathcal A$  yields a polyhedral expansion of both of $X$ and the limit space $\varprojlim {\mathcal F}({\mathbf G})$.
\end{rem}

\begin{question}
For what class of space $X$ and for which cell structure $\mathbf G$ representing $X$ do we have a polyhedral expansion $\varprojlim {\mathcal F}({\mathbf G}) \to {\mathcal F}({\mathbf G})$ with a shape/homotopy equivalence $\varprojlim {\mathcal F}({\mathbf G}) \to X$?
\end{question}

For a space $X$ for which the above question has a positive answer, we would have a way to study the shape type of $X$ by means of cell structures representing $X$.

\end{document}